\documentclass[11pt,english]{article}
\usepackage[T1]{fontenc}
\usepackage[latin9]{inputenc}
\usepackage{geometry}
\usepackage{babel}
\usepackage{mathrsfs}
\usepackage{amsmath}
\usepackage{amsthm}
\usepackage{amssymb}
\usepackage{esint}
\usepackage[unicode=true]
 {hyperref}

\makeatletter
\newcommand{\lyxaddress}[1]{
\par {\raggedright #1
\vspace{1.4em}
\noindent\par}
}
\theoremstyle{plain}
\newtheorem{thm}{\protect\theoremname}
  \theoremstyle{plain}
  \newtheorem{lem}[thm]{\protect\lemmaname}
  \theoremstyle{plain}
  \newtheorem{prop}[thm]{\protect\propositionname}
  \theoremstyle{remark}
  \newtheorem{rem}[thm]{\protect\remarkname}

\makeatother

  \providecommand{\lemmaname}{Lemma}
  \providecommand{\propositionname}{Proposition}
  \providecommand{\remarkname}{Remark}
\providecommand{\theoremname}{Theorem}

\begin{document}
\global\long\def\N{\mathbb{N}}

\global\long\def\Nbar{\bar{\N}}

\global\long\def\R{\mathbb{R}}

\global\long\def\B{\mathcal{B}}

\global\long\def\C{\mathcal{C}}

\global\long\def\Ck{\C_{\mathrm{K}}}

\global\long\def\Cnull{\C_{\mathrm{0}}}

\global\long\def\b{\mathrm{b}}

\global\long\def\S{\boldsymbol{S}}

\global\long\def\indicator{\boldsymbol{1}}

\global\long\def\Dom{\mathcal{D}}

\global\long\def\Range{\mathcal{R}}

\global\long\def\P{\mathrm{P}}

\global\long\def\E{\mathrm{E}}

\global\long\def\d{\mathrm{d}}

\title{Intertwining of the Wright-Fisher diffusion}

\author{Tobi\'{a}\v{s} Hudec\\
e-mail: \href{mailto:tobias.hudec@gmail.com}{tobias.hudec@gmail.com}}

\maketitle

\lyxaddress{Institute of Information Theory and Automation of the CAS\\
Pod Vod\'{a}renskou v\v{e}\v{z}\'{i} 4\\
CZ-182 08\\
Prague 8\\
Czech Republic}
\begin{abstract}
It is known that the time until a birth and death process reaches
a certain level is distributed as a sum of independent exponential
random variables. Diaconis, Miclo and Swart gave a probabilistic proof
of this fact by coupling the birth and death process with a pure birth
process such that the two processes reach the given level at the same
time. Their coupling is of a special type called intertwining of Markov
processes. We apply this technique to couple the Wright-Fisher diffusion
with reflection at $1/2$ and a pure birth process. We show that in
our coupling the time of absorption of the diffusion is a.\,s.\ equal
to the time of explosion of the pure birth process. The coupling also
allows us to interpret the diffusion as being initially reluctant
to get absorbed, but later getting more and more compelled to get
absorbed.\end{abstract}
\begin{verse}

\textbf{Classification}: 60J60, 60J35, 60J27.
\end{verse}

\section{Introduction and the main result}

\subsection{Introduction}

It is known that the time until a birth and death process $X_{t}$
started at the origin reaches a certain level is distributed as a
sum of independent exponential variables whose parameters are the
negatives of the non-zero eigenvalues of the generator of the process
stopped at the given level (see Karlin~\cite{karlin1959coincidence}).
Diaconis and Miclo~\cite{diaconisMiclo2009quasi-stationarity} and
Swart~\cite{swart2010intertwining} gave a probabilistic proof of
this fact by finding a pure birth process $Y_{t}$ which reaches the
given level at the same time as $X_{t}$. The technique that Diaconis,
Miclo and Swart employ is called intertwining of Markov processes.
This technique was developed by Rogers and Pitman~\cite{rogersPitman1981markov},
Diaconis and Fill~\cite{diaconisFill90strong-stationary-times},
and Fill~\cite{fill92strong-stationary-duality}. It allows them
to add structure to the process $X_{t}$ such that it is initially
reluctant to be absorbed, but after each exponential time (which corresponds
to jump times of $Y_{t}$) it changes its behavior to be more and
more compelled to be absorbed. Since one-dimensional diffusions can
be obtained as limits of birth and death processes, it is interesting
to investigate whether this technique can be extended to the case
that $X_{t}$ is a diffusion. Since the general case is too difficult,
in this paper we consider the case that $X_{t}$ is the Wright-Fisher
diffusion with reflection at $1/2$, which has state-space $[1/2,1]$
and is absorbed at $1$. The generator and the semigroup of this diffusion
have the nice property that they map polynomials to polynomials of
the same order, which simplifies our proofs. We need that the diffusion
is reflected at $1/2$ for technical reasons; without it, one of our
proofs would not work (see Remark~\ref{rem:WF-fail}). We construct
an explosive pure birth process $Y_{t}$ such that $X_{t}$ is absorbed
at the same time as $Y_{t}$ explodes.

The idea of Diaconis, Miclo and Swart can be summarized as follows.
For a given transition semigroup $P_{t}$ of a birth and death process
$X_{t}$ on $\left\{ 0,\dots,n\right\} $ absorbed at $n$, Swart
finds a transition semigroup $Q_{t}$ of a pure birth process $Y_{t}$
on $\left\{ 0,\dots,n\right\} $ and a probability kernel $K$ which
satisfies
\begin{equation}
P_{t}K=KQ_{t}\qquad(t\ge0).\label{eq:intr-PK=00003DKQ}
\end{equation}
The algebraic relation~(\ref{eq:intr-PK=00003DKQ}) is called intertwining,
which gives the name to the intertwining of Markov processes. Swart
builds on an earlier work of Diaconis and Miclo, who found an intertwining
of the form $KP_{t}=Q_{t}K$. However, we focus on Swart's construction,
because our work on the Wright-Fisher diffusion is closer in spirit
to his. He uses a result proved by Fill~\cite{fill92strong-stationary-duality}
which says that if $X_{t}$ and $Y_{t}$ are Markov processes with
finite state-spaces related by~(\ref{eq:intr-PK=00003DKQ}), then
the two processes can be coupled (i.\,e.\ defined on the same probability
space) such that\footnote{Under certain conditions, Fill showed that this coupling can be extended
to countably infinite state-spaces. Fill built on earlier work of
Diaconis and Fill~\cite{diaconisFill90strong-stationary-times},
where an analogous result is proved for processes with discrete time.}
\begin{equation}
\P\left(Y_{t}=y|X_{u},0\leq u\leq t\right)=K\left(X_{t},y\right)\qquad\mathrm{a.s.}\quad(t\geq0).\label{eq:intr-coupling}
\end{equation}
Using (\ref{eq:intr-coupling}) and the fact that his kernel satisfies
\begin{equation}
K(x,n)=\indicator_{[x=n]}:=\begin{cases}
1, & \mathrm{if}\,x=n,\\
0, & \mathrm{otherwise,}
\end{cases}\label{eq:intr-K}
\end{equation}
Swart proves that $X_{t}$ and $Y_{t}$ can be coupled such that the
times of absorption of $X_{t}$ and $Y_{t}$ are a.\,s.\ the same. 

In this paper, we derive analogue results for the case that $X_{t}$
is the Wright-Fisher diffusion with reflection at zero. We find an
explosive pure-birth process $Y_{t}$ on $\mbox{\ensuremath{\bar{\N}}:=}\N\cup\left\{ \infty\right\} :=\left\{ 0,1,\dots,\infty\right\} $
and a probability kernel $K$ from $[0,1]$ to $\bar{\N}$ satisfying
the intertwining relation~(\ref{eq:intr-PK=00003DKQ}) and
\[
K\left(x,\infty\right)=\indicator_{\left[x=1\right]}.
\]
We couple the two processes such that they satisfy~(\ref{eq:intr-coupling})
which allows us to conclude that the time of absorption of $X_{t}$
is a.\,s.\ equal to the time of explosion of $Y_{t}$. But since
the time of explosion of the pure birth process is the sum of independent
exponential variables whose intensities are the birth rates of $Y_{t}$,
this gives us a new proof of the distribution of the time to absorption
of $X_{t}$.\footnote{Just as with the birth and death processes, the distribution of the
time of absorption of the Wright-Fisher diffusion has long been known
(see e.\,g.\ Kent~\cite{kent1982spectral}), but our proof is new.}

\subsection{Intertwining of the Wright-Fisher diffusion}

Define
\[
\Dom(G)=\left\{ f\in\C^{2}[0,1];\,\frac{\partial}{\partial x}f(0)=0\right\} 
\]
and
\begin{equation}
Gf(x)=\left(1-x^{2}\right)\frac{\partial^{2}}{\partial x^{2}}f(x),\qquad f\in\Dom(G),x\in[0,1].\label{eq:WF-generator}
\end{equation}
In the appendix we show that $G$ is closable and its closure generates
a Feller semigroup, which we denote $P_{t}$. We also show that the
associated Markov process, which we call the Wright-Fisher diffusion
with reflection at zero, has continuous sample paths. Note that the
generator of the Wright-Fisher diffusion is usually defined as
\begin{equation}
\frac{1}{2}x(1-x)\frac{\partial^{2}}{\partial x^{2}}f(x),\quad f\in\C^{2}[0,1],\,x\in[0,1]\label{eq:WF-=00005B0,1=00005D-generator}
\end{equation}
(see e.\,g.\ Liggett~\cite[Example 3.48]{liggett2010continuous}).
However, if $\tilde{X}_{t}$ is generated by~(\ref{eq:WF-=00005B0,1=00005D-generator}),
then $X_{t}=\left|2\tilde{X}_{2t}-1\right|$ is generated by~(\ref{eq:WF-generator}).

Define $H$ as the generator of an explosive pure birth process on
$\bar{\N}$ which jumps from $y$ to $y+1$ with the rate 
\[
\lambda_{y}=\left(2y+1\right)\left(2y+2\right),\,y\in\N.
\]
That is, define $H$ as an operator from $\R^{\Nbar}$ to $\R^{\Nbar}$
by
\begin{eqnarray*}
Hf(y) & = & \lambda_{y}\left(f(y+1)-f(y)\right),\quad y\in\N,\\
Hf(\infty) & = & 0,
\end{eqnarray*}
where $f$ is in $\R^{\Nbar}$. It is shown in the Appendix that the
restriction of $H$ to a suitable domain is the generator of a Feller
semigroup on $\C\left(\Nbar\right)$, which we denote by $Q_{t}$.
Define a probability kernel from $[0,1]$ to $\bar{\N}$ by
\begin{equation}
K(x,y)=\begin{cases}
\left(1-x^{2}\right)x^{2y}, & \mathrm{if\,}0\leq y<\infty,\\
\indicator_{[x=1]}, & \mathrm{if\,}y=\infty.
\end{cases}\label{eq:K-def}
\end{equation}
It can be shown that $K$ maps $\C\left(\bar{\N}\right)$ into $\C[0,1]$.
We claim that there is an intertwining relation:
\begin{thm}
\label{thm:PK=00003DKQ}We have
\begin{equation}
P_{t}K=KQ_{t},\qquad(t\geq0).\label{eq:PK=00003DKQ}
\end{equation}

\end{thm}
Using~(\ref{eq:PK=00003DKQ}), we are able to couple the two processes
in the spirit of~\cite{diaconisMiclo2009quasi-stationarity,fill92strong-stationary-duality,swart2010intertwining}.
We define the state-space $\S$ of the coupled process as the one-point
compactification of $[0,1]\times\N$, where we denote the point at
infinity by $(1,\infty)$. Using this notation, we can think of $\S$
as a subset of $[0,1]\times\Nbar$, but keep in mind that the topology
of $\S$ is not the one induced by $[0,1]\times\Nbar$. Note that
if the coupled process is to satisfy an analogue of~(\ref{eq:intr-coupling}),
then it would be natural to construct it on the space
\[
\left\{ (x,y)\in[0,1]\times\Nbar;\,K(x,y)>0\right\} ,
\]
as was done for Markov processes with discrete state-spaces~\cite{diaconisFill90strong-stationary-times,fill92strong-stationary-duality}.
However, this space is not compact, so if we want to use the theory
of Feller semigroups we must compactify it (either implicitly or explicitly).
It turns out that $\S$ is the right compactification.

An analogous result to the following theorem was proved by Fill~\cite[Theorem 2]{fill92strong-stationary-duality}
for processes with discrete state-spaces and by Diaconis and Fill~\cite[Theorem 2.33]{diaconisFill90strong-stationary-times}
for processes with discrete time and space. 
\begin{thm}
\label{thm:coupling}There exists a Feller process $\left(X_{t},Y_{t}\right)$
on $\S$ such that
\begin{equation}
\E\left(f\left(Y_{s+t}\right)|X_{u},Y_{u},0\leq u\leq s\right)=\left(Q_{t}f\right)\left(Y_{s}\right)\,\mathrm{a.s.}\label{eq:second-margin}
\end{equation}
for all $f\in\C\left(\bar{\N}\right)$ and $s,t\geq0$. Hence, $Y_{t}$
on its own is a pure birth process on $\bar{\N}$ with birth rates
$\lambda_{y}$. If the initial distribution satisfies
\begin{equation}
\pi_{0}^{(X,Y)}(A\times\{y\})=\int_{A}K(x,y)\pi_{0}^{X}(\d y),\label{eq:coupling-initial}
\end{equation}
where $\pi_{0}^{X}$ is an arbitrary probability measure on $[0,1]$,
then $X_{t}$ on its own is the Wright-Fisher diffusion with reflection
at zero with initial distribution $\pi_{0}^{X}$ and we have
\begin{equation}
\P\left(Y_{t}=y|X_{s},0\leq s\leq t\right)=K\left(X_{t},y\right)\quad\mathrm{a.s.}\label{eq:averaging}
\end{equation}
for all $y\in\bar{\N}$ and $t\geq0$.
\end{thm}
Note that if both $X_{t}$ and $Y_{t}$ start from zero, then~(\ref{eq:coupling-initial})
is satisfied. Using Theorem~\ref{thm:coupling} we can prove that
the time of absorption of the diffusion is a.\,s.\ equal to the
time of explosion of the pure birth process. Indeed, from~(\ref{eq:averaging})
we have
\[
\P\left(X_{t}\in A,Y_{t}\in B\right)=\E\left(\indicator_{A}\left(X_{t}\right)K\left(X_{t},B\right)\right).
\]
But since $K(x,\cdot)$ is concentrated on $\N$ for $x<1$ and on
$\{\infty\}$ for $x=1$, we have
\[
\P\left(X_{t}<1,Y_{t}=\infty\right)=\P\left(X_{t}=1,Y_{t}<\infty\right)=0.
\]

\subsection{Discussion}

In addition to proving the a.\,s.\ equality of the time of absorption
and the time of explosion of the two processes, Theorem~\ref{thm:coupling}
shows more about the underlying structure. By inspection of the formula
for the coupled generator~(\ref{eq:coupled-generator-explicit})
below we see that conditionally on $Y_{t}=y\in\N$ we can interpret
$X_{t}$ as the Wright-Fisher diffusion with reflection at zero and
with additional drift, which at the point $X_{t}=x$ equals
\[
\frac{4y}{x}-4(y+1)x.
\]
It can be shown that the scale function $u(x)$ and the speed measure
$m(\d x)$ of this diffusion satisfy
\begin{eqnarray*}
u'(x) & = & \frac{1}{x^{4y}\left(1-x^{2}\right)^{2}},\\
m(\d x) & = & x^{4y}\left(1-x^{2}\right)\d x.
\end{eqnarray*}
Hence, by Mandl~\cite[pp.~24--25]{mandl1968analytical}, both boundaries
are entrance for $y>0$ and $0$ is a regular boundary while $1$
is an entrance boundary for $y=0$. In particular, the coupled process
lives on the set $\left\{ (x,y)\in[0,1]\times\Nbar;\,K(x,y)>0\right\} $
as one might expect.

Moreover, we can see that there is an equilibrium point
\[
x_{y}=\sqrt{\frac{y}{y+1}}
\]
such that the drift is positive when $x<x_{y}$ and negative when
$x>x_{y}$. We can interpret this as, conditionally on $Y_{t}=y$,
$X_{t}$ is pushed toward the equilibrium point $x_{y}$. Obviously,
$x_{y}$ is monotonous in $y$ and goes from $0$ to $1$ as $y$
goes from $0$ to $\infty$. Thus, we can think of $X_{t}$ as being
initially reluctant to be absorbed, but later getting more and more
compelled to get absorbed.

In our paper we construct Markov processes from generators using the
Hille-Yosida theorem. We could also construct them as solutions to
Martingale problems or stochastic differential equations. However,
we chose the Hille-Yosida theorem for its simplicity. Theorem~\ref{thm:PK=00003DKQ}
extends results of Diaconis and Miclo~\cite{diaconisMiclo2009quasi-stationarity}
and Swart~\cite{swart2010intertwining}, who proved similar theorems
for birth and death processes. Theorem~\ref{thm:coupling} extends
results of Fill~\cite{fill92strong-stationary-duality}, who proved
similar result for Markov processes with continuous time and discrete
state-space, and of Diaconis and Fill~\cite{diaconisFill90strong-stationary-times},
who proved it for the case of discrete time and space. It remains
an open problem whether results like Theorems~\ref{thm:PK=00003DKQ}
and \ref{thm:coupling} hold for other diffusions than the modified
version of the Wright-Fisher diffusion we consider in our paper. It
seems that our proof of Theorem~\ref{thm:PK=00003DKQ} does not exploit
any peculiarity of the Wright-Fisher diffusion and we believe it could
be extended to other types of diffusions as well. On the other hand,
our proof of Theorem~\ref{thm:coupling} depends strongly on the
fact that the generator of the Wright-Fisher diffusion maps polynomials
to polynomials of the same order, and it seems that entirely different
proof techniques would be required for other diffusions.

\section{Proofs}

\subsection{Intertwining}

To prove Theorem~\ref{thm:PK=00003DKQ} we need to show that there
is an intertwining between semigroups $P_{t}$ and $Q_{t}$. The following
theorem says that we can show this by proving that there is an intertwining
between the generators. An analogous result for Markov processes with
discrete state-spaces was proved by Fill~\cite[Lemma 3]{fill92strong-stationary-duality}.
Although we use the following theorem only when $P_{t}$ and $Q_{t}$
are Feller semigroups and $K$ is a probability kernel, we are able
to prove it more generally.
\begin{thm}
\label{thm:Semigroup-intertwining}Let $L_{1},L_{2}$ be Banach spaces.
Let $P_{t}$ and $Q_{t}$ be strongly continuous contraction semigroups
defined on $L_{1},L_{2}$ and let $G$ and $H$ be their generators.
Let $K:L_{2}\rightarrow L_{1}$ be a continuous linear operator. Then
the following are equivalent:
\begin{enumerate}
\item \label{enu:SI-semigroup}For all $t\geq0$,
\begin{equation}
P_{t}K=KQ_{t}\label{eq:semigroup-intertwining}
\end{equation}
on $L_{2}$,
\item \label{enu:SI-generator}$K$ maps $\Dom(H)$ into $\Dom(G)$ and
\begin{equation}
GK=KH\label{eq:generator-intertwining}
\end{equation}
on $\Dom(H)$,
\item \label{enu:SI-core}There exists a core $D$ of $H$ (i.\,e.\ $D$
is a dense subspace of $\Dom(H)$ such that the closure of the restriction
of $H$ to $D$ is $H$) such that $K$ maps $D$ into $\Dom(G)$
and (\ref{eq:generator-intertwining}) holds on $D$.
\end{enumerate}
\end{thm}
\begin{proof}
To prove $\eqref{enu:SI-semigroup}\Rightarrow\eqref{enu:SI-generator}$,
fix $f\in\Dom(H)$. Then $\frac{1}{t}\left(Q_{t}f-f\right)$ converges
to $Hf$, so by the continuity of $K$, $\frac{1}{t}\left(KQ_{t}f-Kf\right)$
converges to $KHf$. By~(\ref{eq:semigroup-intertwining}), $\frac{1}{t}\left(P_{t}Kf-Kf\right)$
is also convergent, so $Kf$ is in $\Dom(G)$ and $GKf=KHf$.

In order to prove $\eqref{enu:SI-generator}\Rightarrow\eqref{enu:SI-semigroup}$,
fix $f\in\Dom(H)$ and define $u(t)=KQ_{t}f$. Since $Q_{t}f$ is
in $\Dom(H)$ by~\cite[Proposition 1.1.5]{EthierKurz}, $u(t)\in\Dom(G)$
for all $t\geq0$. By the continuity of $K$
\[
\frac{\d}{\d t}u(t)=K\frac{\d}{\d t}Q_{t}f=KHQ_{t}f=GKQ_{t}f=Gu(t).
\]
Since
\[
\frac{\d}{\d t}u(t)=KQ_{t}Hf,
\]
$Gu(t)=\frac{\d}{\d t}u(t)$ is continuous as a function of $t$.
By Proposition~1.3.4 in Ethier and Kurtz~\cite{EthierKurz}, $u(t)=P_{t}u(0)=P_{t}Kf$
which proves that (\ref{eq:semigroup-intertwining}) holds on $\Dom(H)$.
Since all operators involved in~(\ref{eq:semigroup-intertwining})
are continuous, the assertion now follows from the density of $\Dom(H)$
in $L_{2}$.

The implication $\eqref{enu:SI-generator}\Rightarrow\eqref{enu:SI-core}$
is trivial by taking $D=\Dom(H)$. To prove the converse, let $f$
be in $\Dom(H)$. Then there exist $f_{n}\in D$ such that $f_{n}\rightarrow f$
and $Hf_{n}\rightarrow Hf$. Since $K$ is continuous, $Kf_{n}\rightarrow Kf$
and $GKf_{n}=KHf_{n}\rightarrow KHf$, where we have used (\ref{eq:generator-intertwining})
for $f_{n}$. Since $G$ is a closed operator, $Kf$ is in $\Dom(G)$
and $GKf=KHf$.
\end{proof}
Theorem~\ref{thm:Semigroup-intertwining} shows that it suffices
to prove~(\ref{eq:generator-intertwining}) on a core of $H$. In
the Appendix it is shown that 

\begin{equation}
D_{H}=\left\{ f\in\C\left(\bar{\N}\right);\,\exists y_{0}\in\N\,\mathrm{s.\,t.}\,\forall y>y_{0}\,f(y)=f(\infty)\right\} \label{eq:D_H}
\end{equation}
is a core of $H$. The following theorem verifies condition~\ref{enu:SI-core}
of Theorem~\ref{thm:Semigroup-intertwining}.
\begin{thm}
\label{thm:GK=00003DKH}$K$ maps $D_{H}$ into $\Dom(G)$ and
\[
GK=KH
\]
on $D_{H}$.\end{thm}
\begin{proof}
Fix $f$ in $D$ and let $y_{0}\in\N$ be such that $f(y)=f(\infty)$
for all $y\geq y_{0}$. Then for $x\in[0,1)$,
\begin{eqnarray}
Kf(x) & = & \left(1-x^{2}\right)\sum_{y=0}^{y_{0}-1}x^{2y}f(y)+x^{2y_{0}}f\left(y_{0}\right)\nonumber \\
 & = & f(0)+\sum_{y=1}^{y_{0}}x^{2y}\left(f(y)-f(y-1)\right).\label{eq:Kf}
\end{eqnarray}
As $x$ approaches $1$, $Kf(x)$ approaches $f\left(y_{0}\right)$.
Now
\[
Kf\left(1\right)=f(\infty)=f\left(y_{0}\right).
\]
Hence (\ref{eq:Kf}) holds also for $x=1$, and therefore $Kf$ is
in $\C^{\infty}[0,1]\subseteq\C^{2}(G)$. Moreover
\[
\frac{\partial}{\partial x}Kf(x)=\sum_{y=1}^{y_{0}}2yx^{2y-1}\left(f(y)-f(y-1)\right),
\]
hence
\[
\frac{\partial}{\partial x}Kf(0)=0.
\]
We have shown that $Kf$ is in $\Dom(G)$.

From~(\ref{eq:Kf}) we have that for $x\in[0,1]$,
\begin{eqnarray*}
GKf(x) & = & \left(1-x^{2}\right)\sum_{y=1}^{y_{0}}2y(2y-1)x^{2y-2}\left(f(y+1)-f(y)\right)\\
 & = & \left(1-x^{2}\right)\sum_{y=0}^{y_{0}-1}\lambda_{y}x^{2y}\left(f(y+1)-f(y)\right).
\end{eqnarray*}
Now for $y<y_{0}$
\[
Hf(y)=\lambda_{y}\left(f(y+1)-f(y)\right),
\]
and for $y\geq y_{0}$,
\[
Hf(y)=0,
\]
hence for $x\in[-1,1]$,
\[
KHf(x)=\left(1-x^{2}\right)\sum_{y=0}^{y_{0}-1}\lambda_{y}x^{2y}\left(f(y+1)-f(y)\right).
\]

\end{proof}

\begin{proof}[Proof of Theorem~\ref{thm:PK=00003DKQ}.]
We use Theorem~\ref{thm:Semigroup-intertwining}. In the present
context, $L_{1}=\C[0,1]$ and $L_{2}=\C\left(\bar{\N}\right)$. Thus
we need to show that $K$ maps $\C\left(\Nbar\right)$ to $\C[0,1]$.
This is equivalent to saying that the measures $K(x,\cdot)$ are continuous
in $x$ with respect to the weak convergence. But this is easy to
prove, since $K(x,\cdot)$ is geometric distribution with success
parameter $1-x^{2}$ if $x<1$, and it is the degenerate distribution
$\delta_{\infty}$ if $x=1$. Theorem~\ref{thm:GK=00003DKH} verifies
condition~\ref{enu:SI-core} of Theorem~\ref{thm:Semigroup-intertwining}.
Theorem~\ref{thm:Semigroup-intertwining} shows that this is equivalent
to condition~\ref{enu:SI-semigroup}, and this is what we had to
prove.
\end{proof}

\subsection{Coupling}

In order to find a coupling of Theorem~\ref{thm:coupling}, recall
that $\S$ is the one-point compactification of $[0,1]\times\N$,
where $\left(1,\infty\right)$ denotes the point at infinity. It is
easy to see that $f:\S\rightarrow\R$ is continuous if and only if
$f\left(\cdot,y\right)$ is continuous for all $y\in\N$ and $f(x,y)\rightarrow f(1,\infty)$
as $y\rightarrow\infty$, uniformly in $x$. It is also easy to see
that
\[
\left\{ f\in\C(\S);f(x,y)=f(1,\infty)\,\mathrm{for\,all}\,x\in[0,1]\,\mathrm{and}\,y>y_{0}\,\mathrm{for\,some}\,y_{0}\in\N\right\} 
\]
is dense in $\C(\S)$. Since even polynomials are dense in $\C[0,1]$
by the Stone-Weierstrass theorem, it follows that
\begin{equation}
\begin{aligned}\Dom\left(\boldsymbol{G}\right)= & \left\{ f\in\C(\S);\exists y_{0}\in\N\,\mathrm{s.\,t.}\,f(x,y)=f(1,\infty)\,\mathrm{for\,all}\,x\in[0,1]\,\mathrm{and}\,y>y_{0},\right.\\
 & \left.f(\cdot,y)\,\mathrm{is\,an\,even\,polynomial\,for\,all}\,y\leq y_{0}\right\} 
\end{aligned}
\label{eq:domain}
\end{equation}
is dense in $\C(\S)$. 

We now define an operator $\boldsymbol{G}$ with domain $\Dom(\boldsymbol{G})$
which we later prove generates a Feller process satisfying Theorem~\ref{thm:coupling}.
For the motivation of this definition, see section~\ref{sub:Generator-derivation}.
For $f\in\Dom(\boldsymbol{G})$ and $(x,y)\in(0,1)\times\N$ define
\begin{equation}
\boldsymbol{G}f(x,y)=\left(Hf\left(x,\cdot\right)\right)(y)+\frac{\left(Gf\left(\cdot,y\right)K\left(\cdot,y\right)\right)(x)-f(x,y)\left(GK\left(\cdot,y\right)\right)(x)}{K(x,y)}.\label{eq:coupled-generator-general}
\end{equation}
Here $\left(Gf\left(\cdot,y\right)K\left(\cdot,y\right)\right)$ denotes
the application of the operator $G$ to the product of $f\left(\cdot,y\right)$
and $K(\cdot,y)$, and $\left(GK\left(\cdot,y\right)\right)$ is the
application of the operator $G$ to $K(\cdot,y)$. In both cases,
$y$ is held fixed, so $f(\cdot,y)$ and $K(\cdot,y)$ are viewed
as functions of $x$ only. $\left(Hf\left(x,\cdot\right)\right)$
is interpreted similarly, but here $x$ is held fixed. Note that $f(\cdot,y)K(\cdot,y)$
and $K(\cdot,y)$ are even polynomials, hence they are in $\Dom(G)$.
Moreover, $f(x,\cdot)$ is in $D_{H}$ of~(\ref{eq:D_H}), which
is a core of $H$ as shown in the Appendix, hence $f(x,\cdot)$ is
in $\Dom(H)$. Finally, $K(x,y)>0$ since $x$ is in $(0,1)$ and
$y<\infty$. Therefore, all the expressions in~(\ref{eq:coupled-generator-general})
are well defined. After plugging in the definitions of $H$ and $G$,
we can get an explicit formula for $\boldsymbol{G}$.
\begin{lem}
\label{lem:coupled-generator-explicit}Let $f$ be in $\Dom(\boldsymbol{G})$
and $(x,y)\in(0,1)\times\N$. Let $\boldsymbol{G}f$ be defined by~(\ref{eq:coupled-generator-general}).
Then
\begin{equation}
\boldsymbol{G}f(x,y)=\lambda_{y}\left(f(x,y+1)-f(x,y)\right)+\left(1-x^{2}\right)\frac{\partial^{2}}{\partial x^{2}}f(x,y)+4\left[\frac{y}{x}-(y+1)x\right]\frac{\partial}{\partial x}f(x,y).\label{eq:coupled-generator-explicit}
\end{equation}
\end{lem}
\begin{proof}
Observe that for $(x,y)\in(0,1)\times\N$,
\begin{eqnarray*}
\left(Gf\left(\cdot,y\right)K\left(\cdot,y\right)\right)(x) & = & \left(1-x^{2}\right)\left(\frac{\partial^{2}}{\partial x^{2}}f(x,y)\right)K(x,y)\\
 &  & +2\left(1-x^{2}\right)\frac{\partial}{\partial x}f(x,y)\frac{\partial}{\partial x}K(x,y)\\
 &  & +\left(1-x^{2}\right)f(x,y)\frac{\partial^{2}}{\partial x^{2}}K(x,y).
\end{eqnarray*}
Hence
\begin{multline*}
\frac{\left(Gf\left(\cdot,y\right)K\left(\cdot,y\right)\right)(x)-f(x,y)\left(GK\left(\cdot,y\right)\right)(x)}{K(x,k)}\\
\begin{aligned}= & \left(1-x^{2}\right)\frac{\partial^{2}}{\partial x^{2}}f(x,y)+2\left(1-x^{2}\right)\frac{\frac{\partial}{\partial x}K(x,y)}{K(x,y)}\frac{\partial}{\partial x}f(x,y).\end{aligned}
\end{multline*}
Noting that
\begin{eqnarray*}
2\left(1-x^{2}\right)\frac{\frac{\partial}{\partial x}K(x,y)}{K(x,y)} & = & 2\frac{2yx^{2y-1}-(2y+2)x^{2y+1}}{y^{2k}}\\
 & = & \frac{4y}{x}-4(y+1)x,
\end{eqnarray*}
we get
\begin{multline}
\frac{\left(Gf\left(\cdot,y\right)K\left(\cdot,y\right)\right)(x)-f(x,y)\left(GK\left(\cdot,y\right)\right)(x)}{K(x,k)}\\
=\left(1-x^{2}\right)\frac{\partial^{2}}{\partial x^{2}}f(x,y)+4\left[\frac{y}{x}-(y+1)x\right]\frac{\partial}{\partial x}f(x,y).\label{eq:coupled-generator-part}
\end{multline}
Plugging~(\ref{eq:coupled-generator-part}) and the definition of
$H$ into~(\ref{eq:coupled-generator-general}), we get~(\ref{eq:coupled-generator-explicit}).
\end{proof}
Formula~(\ref{eq:coupled-generator-explicit}) is well defined even
for $x=1$. Moreover, since $f(x,y)$ is an even polynomial in $x$,
it follows that
\[
\lim_{x\downarrow0}\frac{y}{x}\frac{\partial}{\partial x}f(x,y)
\]
exists, hence we can define $\boldsymbol{G}f$ on $[0,1]\times\N$
by taking the limit. Observe that for $y>y_{0}$ (where $y_{0}$ is
as in~(\ref{eq:domain})), $\boldsymbol{G}f(x,y)=0$. Therefore,
if we define $\boldsymbol{G}f(1,\infty)=0$, then $\boldsymbol{G}f$
is in $\Dom\left(\boldsymbol{G}\right)\subseteq\C(\S)$ and we can
view $\boldsymbol{G}:\Dom(\boldsymbol{G})\rightarrow\C(\S)$ as a
linear operator.
\begin{thm}
\label{thm:generator}Operator $\boldsymbol{G}$ is closable and its
closure generates a Feller semigroup.
\end{thm}
In order to prove Theorem~\ref{thm:generator}, we use the following
corollary to the Hille-Yosida theorem.
\begin{prop}
\label{prop:Hille-Yosida-subspaces}Let $E$ be a compact metric space,
$\Dom(G)$ a subspace of $\C(E)$, and $G:\Dom(G)\rightarrow\C(E)$
a linear operator. Suppose that $1$ is in $\Dom(G)$ and $G1=0$,
$G$ satisfies the positive maximum principle, and there exist a sequence
$\left(L_{n}\right)_{n\in\N}$ of finite-dimensional subspaces of
$\Dom(G)$ such that $\bigcup_{n\in\N}L_{n}$ is dense in $\C(E)$
and $G:L_{n}\rightarrow L_{n}$. Then $G$ is closable and its closure
generates a Feller semigroup.\end{prop}
\begin{proof}
Lemma~4.2.1 in Ethier and Kurtz~\cite{EthierKurz} shows that $G$
is dissipative, and Proposition~1.3.5 in~\cite{EthierKurz} then
proves that $G$ is closable and its closure generates a strongly
continuous contraction semigroup. Finally, the fact that $G1=0$ proves
that $G$ is conservative.
\end{proof}

\begin{proof}[Proof of Theorem~\ref{thm:generator}]
Let us first prove that $\boldsymbol{G}$ satisfies the positive
maximum principle. Let $f\in D$ and $\left(x_{0},y_{0}\right)\in\S$
be such that
\[
\sup_{(x,y)\in\S}f(x,y)=f\left(x_{0},y_{0}\right)\geq0.
\]
First if $y_{0}=\infty$, then $\boldsymbol{G}f\left(\boldsymbol{z}_{0}\right)=0$
by definition. Second, let us assume that $\left(x_{0},y_{0}\right)\in[0,1]\times\N$.
Then we have $f\left(x_{0},y_{0}+1\right)-f\left(x_{0},y_{0}\right)\leq0$.
If $x_{0}\in(0,1)$, then $\frac{\partial}{\partial x}f\left(x_{0},y_{0}\right)=0$
and $\frac{\partial^{2}}{\partial x^{2}}f\left(x_{0},y_{0}\right)\leq0$,
hence $\boldsymbol{G}f\left(x_{0},y_{0}\right)\leq0$. If $x_{0}=1$,
then
\[
4\left[\frac{y_{0}}{x_{0}}-(y_{0}+1)x_{0}\right]\frac{\partial}{\partial x}f(x_{0},y_{0})\leq0
\]
and
\[
\left(1-x_{0}^{2}\right)\frac{\partial^{2}}{\partial x^{2}}f\left(x_{0},y_{0}\right)=0,
\]
so $\boldsymbol{G}f\left(x_{0},y_{0}\right)\leq0$. And if $x_{0}=0$,
then the second-order term of the polynomial $f\left(\cdot,y_{0}\right)$
must be non-positive, for otherwise $\left(0,y_{0}\right)$ could
not be a point of maximum. Hence
\[
\lim_{x\downarrow0}\frac{y_{0}}{x}\frac{\partial}{\partial x}f\left(x,y_{0}\right)\leq0
\]
and
\[
\left(1-x_{0}^{2}\right)\frac{\partial^{2}}{\partial x^{2}}f\left(x_{0},y_{0}\right)\leq0,
\]
so we again get that $\boldsymbol{G}f\left(x_{0},y_{0}\right)\leq0$.
We have shown that $\boldsymbol{G}$ satisfies the positive maximum
principle.

Define
\begin{eqnarray*}
L_{n} & = & \left\{ f\in\C(\S);\,f(\cdot,y)\,\mathrm{is\,an\,even\,polynomial\,of\,degree\,at\,most}\,2n\,\mathrm{for\,all}\,y\leq n,\right.\\
 &  & \left.f(\cdot,y)=f(1,\infty)\,\mathrm{for\,all}\,y>n\right\} .
\end{eqnarray*}
It is easy to see that $\boldsymbol{G}:L_{n}\rightarrow L_{n}$ and
$\bigcup_{n\in\N}L_{n}=\Dom(\boldsymbol{G})$ is dense in $\C[0,1]$.
Finally, $\boldsymbol{G}1=0$, so by Proposition~\ref{prop:Hille-Yosida-subspaces},
$\boldsymbol{G}$ is closable and its closure generates a Feller semigroup.\end{proof}
\begin{rem}
\label{rem:WF-fail}The proof of Theorem~\ref{thm:generator} is
the only place where our argument fails for the Wright-Fisher diffusion
without reflection at zero. Indeed, we could take $P_{t}$ to be the
semigroup of the diffusion on the whole interval $[-1,1]$, that is,
$P_{t}$ would be generated by
\begin{equation}
Gf(x)=\left(1-x^{2}\right)\frac{\partial^{2}}{\partial x^{2}}f(x),\qquad f\in\C^{2}[-1,1],x\in[-1,1].\label{eq:WF-=00005B-1,1=00005D-generator}
\end{equation}
We could now extend kernel $K$ to be from $[-1,1]$ to $\bar{\N}$
using the same formula~(\ref{eq:K-def}). Our proof of Theorem~\ref{thm:PK=00003DKQ}
would still work. We could define $\boldsymbol{G}$ by~(\ref{eq:coupled-generator-general})
where $G$ would now be defined by~(\ref{eq:WF-=00005B-1,1=00005D-generator}).
But now we could not take $\Dom(\boldsymbol{G})$ to be functions
such that $f(\cdot,y)$ are even polynomials, because they are not
dense in $\C[-1,1]$. But if we allowed all polynomials, then for
$y>0$ we could not extend $\left(\boldsymbol{G}f\right)(\cdot,y)$
to a continuous function on $[-1,1]$ because of the term $\frac{y}{x}\frac{\partial}{\partial x}f(x,y)$
in~(\ref{eq:coupled-generator-explicit}). The deeper reason for
this problem is that $K(0,y)=0$ for $y>0$, so if the process $\left(X_{t},Y_{t}\right)$
satisfies~(\ref{eq:averaging}), then after $Y_{t}$ departs from
zero, $X_{t}$ is no longer allowed to cross zero, so the behavior
of the diffusion on $[-1,0]$ and $[0,1]$ are independent. To overcome
this problem, we could define
\[
\S=[-1,1]\times\{0\}\cup\left[-1,0^{-}\right]\times\{1,2,\dots\}\cup\left[0^{+},1\right]\times\{1,2\dots\}\cup\{-1,1\}\times\{\infty\},
\]
where we think of $0^{-}$ and $0^{+}$ as two different points. Now
we could take $\Dom(\boldsymbol{G})$ to be functions such that $f(\cdot,0)$
is a polynomial, $f(\cdot,y)$ is an even polynomial with possibly
different coefficients on $\left[-1,0^{-}\right]$ and on $\left[0^{+},1\right]$
and from some $y_{0}$, $f(x,y)$ equals either $f(-1,\infty)$ or
$f(1,\infty)$ depending on whether $x$ is in $\left[-1,0^{-}\right]$
or $\left[0^{+},1\right]$. This set is dense in $\C(\S)$. Then,
however, $\left(Gf\right)(\cdot,0)$ could be discontinuous at $x=0$
because of the term $\lambda_{y}\left(f(x,y+1)-f(x,y)\right)$ in~(\ref{eq:coupled-generator-explicit}).
To get around this problem, we decided to work with the Wright-Fisher
diffusion with reflection at zero.
\end{rem}
In order to prove Theorem~\ref{thm:coupling}, we need the following
theorem due to Rogers and Pittman~\cite{rogersPitman1981markov}.
\begin{thm}
\label{thm:Markov-function}Let $\left( \boldsymbol{S}, \pmb{\mathscr{S}} \right)$
and $\left(S,\mathscr{S}\right)$ be measurable spaces and let $\phi:\boldsymbol{S}\rightarrow S$
be a measurable transformation. Let $\Lambda$ be a probability kernel
from $S$ to $\boldsymbol{S}$ and define a probability kernel from
$\boldsymbol{S}$ to $S$ by
\[
\Phi f=f\circ\phi.
\]
Let $\boldsymbol{X}_{t}$ be a continuous-time Markov process with
state space $\left( \boldsymbol{S}, \pmb{\mathscr{S}} \right)$, transition
semigroup $\boldsymbol{P}_{t}$ and initial distribution $\boldsymbol{\pi}_{0}=\pi_{0}\Lambda$,
for some distribution $\pi_{0}$ on $S$. Suppose further:
\begin{enumerate}
\item $\Lambda\Phi=I$, the identity kernel on $S$,
\item for each $t\geq0$ the probability kernel $P{}_{t}:=\Lambda\boldsymbol{P}_{t}\Phi$
from $S$ to $S$ satisfies 
\begin{equation}
\Lambda\boldsymbol{P}_{t}=P{}_{t}\Lambda.\label{eq:LP=00003DQL}
\end{equation}

\end{enumerate}
Then $P{}_{t}$ is a transition semigroup on $S$, $\phi\circ\boldsymbol{X}_{t}$
is Markov with transition semigroup $P_{t}$ and the initial distribution
$\pi_{0}$ and
\[
\P\left(\boldsymbol{X}_{t}\in A|\phi\circ\boldsymbol{X}_{s},0\leq s\leq t\right)=\Lambda\left(\phi\circ\boldsymbol{X}_{t},A\right)
\]
a.\,s.\ for all $t\geq0$ and $A \in \pmb{\mathscr{S}}$.\end{thm}
\begin{proof}
Rogers and Pittman~\cite[Theorem 2]{rogersPitman1981markov} proved
this for the case that $\pi_{0}=\delta_{y}$ for some $y\in S$. The
general case follows by integration with respect to $\pi_{0}$.
\end{proof}

\begin{proof}[Proof of Theorem~\ref{thm:coupling}.]

It is intuitively clear from the form of $\boldsymbol{G}$ that $Y_{t}$
on its own is generated by $H$, but here we give a short formal proof.
Note that for $f\in\C\left(\bar{\N}\right)$ and $s,t\geq0$, 
\[
\E\left(f\left(Y_{s+t}\right)|X_{u},Y_{u},0\leq u\leq s\right)=\left(\boldsymbol{P}_{t}\Psi f\right)\left(X_{s},Y_{s}\right),
\]
where $\boldsymbol{P}_{t}$ is the semigroup generated by $\boldsymbol{G}$
and $\Psi$ is a kernel given by $\Psi f=f\circ\psi$ where $\psi(x,y)=y$.
Hence, in order to prove~(\ref{eq:second-margin}) we need to show
that 
\begin{equation}
\boldsymbol{P}_{t}\Psi f=\Psi Q_{t}f\label{eq:second-margin-semigroup}
\end{equation}
for all $f\in\C\left(\bar{\N}\right)$. By Theorem~\ref{thm:Semigroup-intertwining}
it suffices to prove that there exists a core $D_{H}$ of $H$ such
that $\Psi$ maps $D_{H}$ into $\Dom(\boldsymbol{G})$ and
\begin{equation}
\boldsymbol{G}\Psi=\Psi H\label{eq:second-margin-generator-intertwining}
\end{equation}
on $D_{H}$. It is shown in Lemma~\ref{lem:H-core} that
\[
D_{H}=\left\{ f\in\C\left(\bar{\N}\right);\,\exists y_{0}\,\mathrm{s.t.}\,f(y)=f(\infty)\,\mathrm{for\,all}\,y>y_{0}\right\} 
\]
is a core of $H$ and it is easy to see that for $f\in D_{H}$, $\Psi f$
is in $\Dom(\boldsymbol{G})$. Moreover, for $f\in D_{H}$ we have
$\left(Gf(y)K(\cdot,y)\right)(x)=f(y)\left(GK(\cdot,y)\right)(x)$,
so from~(\ref{eq:coupled-generator-general}) it is easy to see that~(\ref{eq:second-margin-generator-intertwining})
holds.

In order to prove the claims about $X_{t}$, we will use Theorem~\ref{thm:Markov-function}.
In the present setting, $S=[0,1]$ and $\phi(x,y)=x$. Define a probability
kernel from $[0,1]$ to $\S$ by
\begin{eqnarray*}
\Lambda(x,A\times\{y\}) & = & \delta_{x}(A)K(x,y)\\
\Lambda(x,\left(1,\infty\right)) & = & K(x,\infty)
\end{eqnarray*}
where $x$ is in $[0,1]$, $A$ is in $\B[0,1]$ and $y\in\N$. In
other words, 
\begin{equation}
\Lambda f(x)=\sum_{0\leq y\leq\infty}K(x,y)f(x,y)\label{eq:lf}
\end{equation}
for $f\in\C(\S)$ and $x\in[0,1]$. Observe that $\pi_{0}^{X}\Lambda=\boldsymbol{\pi}_{0}^{(X,Y)}$.
Also observe that for $f\in\C[0,1]$ and $x\in[0,1]$ we have
\[
\Lambda\Phi f(x)=\sum_{0\leq y\leq\infty}K(x,y)f(x)=f(x),
\]
hence 
\begin{equation}
\Lambda\Phi=I.\label{eq:LPhi=00003DI}
\end{equation}

Let us now prove that 
\begin{equation}
\Lambda\boldsymbol{P}_{t}=P_{t}\Lambda.\label{eq:LP=00003DPL}
\end{equation}
By Theorem~\ref{thm:Semigroup-intertwining}, it suffices to prove
that $\Lambda$ maps $\Dom(\boldsymbol{G})$ into $\C^{2}[0,1]$ and
\begin{equation}
\Lambda\boldsymbol{G}=G\Lambda\label{eq:coupled-intertwining}
\end{equation}
on $\Dom(\boldsymbol{G})$. Let $f$ be in $\Dom(\boldsymbol{G})$.
Then there is $y_{0}$ such that $f(x,y)=f(1,\infty)$ for $y>y_{0}$.
Since we know that $\Lambda\boldsymbol{G}1=G\Lambda1=0$, we may without
loss of generality assume that $f(\infty)=0$. Then 
\[
\Lambda f(x)=\sum_{y=0}^{y_{0}}K(x,y)f(x,y),
\]
which is a polynomial, hence in $\C^{2}[0,1]$.

By~(\ref{eq:coupled-generator-general}) we have for $x\in(0,1)$
and $y\leq y_{0}$ that
\begin{equation}
K(x,y)\left(\boldsymbol{G}f\right)(x,y)=K(x,y)\left(Hf\left(\cdot,x\right)\right)(y)+G\left(f\left(\cdot,y\right)K\left(\cdot,y\right)\right)(x)-f(x,y)\left(GK\left(\cdot,y\right)\right)(x).\label{eq:KG}
\end{equation}
Since both sides of the equality are continuous in $x$, the equality
also holds for all $x\in[0,1]$. Using~(\ref{eq:lf}), (\ref{eq:KG})
and noting that $\left(\boldsymbol{G}f\right)(x,y)=0$ for $y>y_{0}$
by~(\ref{eq:coupled-generator-general}) and $\boldsymbol{G}f(1,\infty)=0$
by definition, we get 
\[
\Lambda\boldsymbol{G}f(x)=\sum_{y=0}^{y_{0}}K(x,y)\left(Hf\left(x,\cdot\right)\right)(y)+\sum_{y=0}^{y_{0}}G\left(f\left(\cdot,y\right)K\left(\cdot,y\right)\right)(x)-\sum_{y=0}^{y_{0}}f(x,y)\left(GK\left(\cdot,y\right)\right)(x).
\]
The second term is just $G\Lambda f$, since $f(x,y)=0$ for $y>y_{0}$.
The first term can be rewritten as
\[
\sum_{y=0}^{y_{0}}K(x,y)\sum_{z=0}^{y_{0}}H(y,z)f(x,z),
\]
where we have again used that $f(x,z)=0$ for $z>y_{0}$. The last
term can be written as
\begin{eqnarray*}
\sum_{y=0}^{y_{0}}f(x,y)\left(GK\indicator_{\{y\}}\right)(x) & = & \sum_{y=0}^{y_{0}}f(x,y)\left(KH\indicator_{\{y\}}\right)(x)\\
 & = & \sum_{y=0}^{y_{0}}f(x,y)\sum_{z=0}^{y_{0}}K(x,z)H(z,y)
\end{eqnarray*}
where in the first equality we have used Theorem~\ref{thm:GK=00003DKH}
and in the second equality we have used that $H$ is an upper triangular
matrix. Therefore, $\Lambda\boldsymbol{G}f=G\Lambda f$.

Finally, from (\ref{eq:LPhi=00003DI}) and (\ref{eq:LP=00003DPL})
we get that $P_{t}=\Lambda\boldsymbol{P}_{t}\Phi$. Thus, we have
verified all requirements of Theorem~\ref{thm:Markov-function}.
It follows that $X_{t}$ is the Wright-Fisher diffusion with reflection
at zero with the initial distribution $\pi_{0}^{X}$ and 
\begin{eqnarray*}
\P\left(Y_{t}=y|X_{s},0\leq s\leq t\right) & = & \Lambda\left(X_{t},[0,1]\times\{y\}\right)\\
 & = & K\left(X_{t},y\right)\,a.s.
\end{eqnarray*}
for $y\in\bar{\N}$.
\end{proof}

\subsection{\label{sub:Generator-derivation}Derivation of the generator for
the coupled process}

In this section we show how formula~(\ref{eq:coupled-generator-general})
for the generator of the coupled process can be derived. Strictly
speaking, this derivation is not necessary since (\ref{eq:coupled-generator-general})
can be taken as a definition (and we therefore choose to make this
derivation informal for the sake of brevity). However, we believe
that this derivation can provide insight into the problem. We use
the technique of Diaconis and Fill~\cite{diaconisFill90strong-stationary-times}
who derived an analogous result for Markov processes with discrete
space and time. Fill~\cite{fill92strong-stationary-duality} then
extended the result to Markov processes with continuous time and discrete
space. In Fill's setting, $P_{t}$ and $Q_{t}$ are transition semigroups
of Markov processes with discrete state-spaces $\mathcal{S}_{1}$
and $\mathcal{S}_{2}$. For a fixed $t>0$, he defines a probability
kernel on 
\[
\mathcal{S}=\left\{ (x,y)\in\mathcal{S}_{1}\times\mathcal{S}_{2};\,K(x,y)>0\right\} .
\]
by
\begin{equation}
\boldsymbol{P}^{(t)}f=Q_{t}\indicator_{\left[P_{t}K\neq0\right]}\frac{P_{t}fK}{P_{t}K}.\label{eq:P^t-discrete-ananlogy-compact}
\end{equation}
In~(\ref{eq:P^t-discrete-ananlogy-compact}), $P_{t}fK$ denotes
the application of semigroup $P_{t}$ to the product of functions
$f$ and $K$ (here, $K$ is not viewed as a kernel but simply as
a function of $x$ and $y$). Although $P_{t}$ normally acts on functions
of $x$, we make it act on functions of $x$ and $y$ by fixing $y$.
Similarly, $P_{t}K$ denotes the application of $P_{t}$ to $K$ (this
can alternatively be interpreted as the composition of the two kernels).
Then we take the pointwise division of $P_{t}fK$ and $P_{t}K$, which
we define to be zero if the denominator is zero. We then get $\boldsymbol{P}^{(t)}f$
as the application of $Q_{t}$ to this function. This time, $x$ is
considered fixed when we apply $Q_{t}$.

It turns out that $\boldsymbol{P}^{(t)}$ does not satisfy the Chapman-Kolmogorov
equations and hence cannot be used to construct the coupled process
directly. However, Fill proves that there exists a generator $\boldsymbol{G}$
on $\mathcal{S}$ (for which he gives an explicit formula) such that
\begin{equation}
\frac{\boldsymbol{P}^{(t)}-I}{t}\rightarrow\boldsymbol{G}\label{eq:Fill-generator}
\end{equation}
as $t\downarrow0$. He then shows the bivariate Markov process associated
with $\boldsymbol{G}$ (with suitable initial distribution) has the
desired properties, i.\,e. its margins on their own are processes
with transition semigroups $P_{t}$ and $Q_{t}$ and satisfy~(\ref{eq:intr-coupling}).

Now we return back to our setting where $P_{t}$ is the semigroup
of the Wright-Fisher diffusion with reflection at zero and $Q_{t}$
is the semigroup of an explosive pure birth process. Note that~(\ref{eq:P^t-discrete-ananlogy-compact})
is not a suitable definition in this case, since $P_{t}$ operates
on continuous functions, but the indicator in (\ref{eq:P^t-discrete-ananlogy-compact})
can introduce discontinuity. To get around this problem, it can be
proved that $\frac{P_{t}fK}{P_{t}K}$ can be extended to a continuous
function (Hudec~\cite[Theorem 3.10]{hudec2016Cascades} proves this
for a slightly different kernel $K$, but his proof can easily be
adapted to our setting). Now we can define $\boldsymbol{P}^{(t)}$
by

\begin{equation}
\boldsymbol{P}^{(t)}f=Q_{t}\frac{P_{t}fK}{P_{t}K}.\label{eq:approximation-of-coupled-semigroup}
\end{equation}
Observe that $P_{t}fK\rightarrow fK$, $P_{t}K\rightarrow K$ and
$\frac{P_{t}fK}{P_{t}K}\rightarrow\frac{fK}{K}=f$ as $t\downarrow0$
(provided we choose $\left(x,y\right)$ such that $K(x,y)>0$). Hence,
\[
\frac{1}{t}\left(\boldsymbol{P}^{(t)}f-f\right)=\frac{1}{t}\left(Q_{t}-1\right)\frac{P_{t}fK}{P_{t}K}+\frac{\frac{1}{t}\left(\left(P_{t}-1\right)fK\right)}{P_{t}K}-f\frac{\frac{1}{t}\left(\left(P_{t}-1\right)K\right)}{P_{t}K}
\]
is expected to converge to
\[
Hf+\frac{GfK}{K}-f\frac{GK}{K}.
\]

\appendix

\section{Wright-Fisher diffusion and an explosive pure birth process}

Recall that $G$ is defined by
\[
Gf(x)=\left(1-x^{2}\right)\frac{\partial^{2}}{\partial x^{2}}f(x)
\]
where $x$ is in $[0,1]$ and $f$ is in $\C^{2}[0,1]$ such that
$\frac{\partial}{\partial x}f(0)=0$.
\begin{thm}
Operator $G$ is closable and its closure generates a Feller semigroup.
Moreover, the associated Markov process has continuous sample paths.\end{thm}
\begin{proof}
In order to prove that $G$ is closable and generates a Feller semigroup,
we will use Proposition~\ref{prop:Hille-Yosida-subspaces}. It is
obvious that $G1=0$. Moreover, if we define $L_{n}$ as the set of
all even polynomials of order at most $2n$, then $\bigcup_{n\in\N}L_{n}$
is dense in $\C[0,1]$ by the Stone-Weierstrass theorem and $G$ maps
$L_{n}$ into $L_{n}$. Finally we prove that $G$ satisfies the positive
maximum principle. Let $f$ be in $\Dom\left(G\right)$ and $x_{0}\in\left[0,1\right]$
be such that $\sup_{x\in\left[0,1\right]}f(x)=f\left(x_{0}\right)$.
If $x_{0}\in\left(0,1\right)$ then $\frac{\partial^{2}f}{\partial x^{2}}\left(x_{0}\right)\leq0$.
If $x_{0}=0$, then $\frac{\partial^{2}f}{\partial x^{2}}\left(x_{0}\right)\leq0$,
since $\frac{\partial}{\partial x}f\left(x_{0}\right)=0$. If $x_{0}=1$
then $1-x_{0}^{2}=0$. In all cases, $Gf\left(x_{0}\right)\leq0$.

To prove that almost all sample paths are continuous, it suffices
to show that for each $x_{0}\in[0,1]$ and $\epsilon>0$ there exists
$f\in\Dom(G)$ such that $f\left(x_{0}\right)=\|f\|$, $\sup_{x\in[0,1]\backslash\left(x_{0}-\epsilon,x_{0}+\epsilon\right)}f(x)<\|f\|$
and $Gf\left(x_{0}\right)=0$ (Ethier and Kurtz~\cite[Proposition~4.2.9 and Remark~4.2.10]{EthierKurz}).
Let $x_{0}\in[0,1]$ and $\epsilon>0$ be given. Define
\[
f(x)=1-\left(x^{2}-x_{0}^{2}\right)^{4}.
\]
Then $f\geq0$ on $[0,1]$ and it attains its unique maximum at $x_{0}$.
Hence
\[
\sup_{x\in[0,1]\backslash\left(x_{0}-\epsilon,x_{0}+\epsilon\right)}f(x)<f\left(x_{0}\right)=\|f\|.
\]
Moreover, 
\[
\frac{\partial^{2}}{\partial x^{2}}f\left(x_{0}\right)=0,
\]
so $Gf\left(x_{0}\right)=0$.
\end{proof}
Recall that $H:\R^{\Nbar}\rightarrow\R^{\Nbar}$ is an operator defined
by
\begin{eqnarray*}
Hf(y) & = & \lambda_{y}\left(f(y+1)-f(y)\right),\quad y\in\N,\\
Hf(\infty) & = & 0,
\end{eqnarray*}
where
\[
\lambda_{y}=\left(2y+1\right)\left(2y+2\right),\quad y\in\N.
\]

\begin{prop}
\label{prop:birth-generator-is-feller} Define $D=\left\{ f\in\C\left(\Nbar\right);\,Hf\in\C\left(\Nbar\right)\right\} $.
Then the restriction of $H$ to $D$ is the generator of a Feller
semigroup on $\C\left(\Nbar\right)$.\end{prop}
\begin{proof}
We verify the conditions of Proposition~\ref{prop:Hille-Yosida-subspaces}.
First note that $H1=0\in\C\left(\Nbar\right)$, hence $1$ is in $D$.
Second, if we define $L_{n}$ as the set of all functions $f\in\C\left(\Nbar\right)$
such that $f(y)=f(\infty)$ for all $y>n$, then $\bigcup_{n\in\N}L_{n}$
is dense in $\C\left(\Nbar\right)$ and $H$ maps $L_{n}$ to $L_{n}$.
Finally, we verify the positive maximum principle. Let $f$ be in
$D$ and $y$ in $\Nbar$ such that
\[
\sup_{y\in\Nbar}f(y)=f\left(y_{0}\right).
\]
If $y_{0}<\infty$, then
\[
Hf\left(y_{0}\right)=\lambda_{y_{0}}\left(f\left(y_{0}+1\right)-f\left(y_{0}\right)\right)\leq0,
\]
and if $y_{0}=\infty$, then $Hf\left(y_{0}\right)=0$, so $H$ satisfies
the positive maximum principle.\end{proof}
\begin{lem}
\label{lem:H-core}The set
\[
D_{H}=\left\{ f\in\C\left(\bar{\N}\right);\,\exists y_{0}\,\mathrm{s.t.}\,f(y)=f(\infty)\,\forall y>y_{0}\right\} 
\]
is a core of $H$.\end{lem}
\begin{proof}
It is easy to see that $D_{H}$ is dense in $\C\left(\Nbar\right)$.
Moreover, since the process associated with $H$ can only jump upward,
the semigroup maps $D_{H}$ into itself. The statement of the lemma
now follows from Proposition~1.3.3 in Ethier and Kurtz~\cite{EthierKurz}.
\end{proof}

\section*{Acknowledgement}

This paper is based on the author's Master thesis~\cite{hudec2016Cascades}.
I would like to thank my thesis supervisor, Jan~M.~Swart, for his
useful comments and suggestions that were helpful in writing the thesis
as well as preparing this paper.

\bibliographystyle{plain}
\bibliography{Intertwining_of_the_Wright-Fisher_Diffusion}

\begin{thebibliography}{10}

\bibitem{diaconisFill90strong-stationary-times}
Persi Diaconis and James~Allen Fill.
\newblock Strong stationary times via a new form of duality.
\newblock {\em The Annals of Probability}, 18(4):1483--1522, 1990.

\bibitem{diaconisMiclo2009quasi-stationarity}
Persi Diaconis and Laurent Miclo.
\newblock On times to quasi-stationarity for birth and death processes.
\newblock {\em Journal of Theoretical Probability}, 22(3):558--586, 2009.

\bibitem{EthierKurz}
Stewart~N. Ethier and Thomas~G. Kurtz.
\newblock {\em Markov Processes: Characterization and Convergence}.
\newblock John Wiley \& Sons, 1986.

\bibitem{fill92strong-stationary-duality}
James~Allen Fill.
\newblock Strong stationary duality for continuous-time markov chains. {Part
  I}: Theory.
\newblock {\em Journal of Theoretical Probability}, 5(1):45--70, 1992.

\bibitem{hudec2016Cascades}
Tobi\'{a}\v{s} Hudec.
\newblock Absorption cascades of one-dimensional diffusions.
\newblock Master's thesis, Charles University in Prague, 2016.
\newblock Available at
  \url{https://is.cuni.cz/webapps/zzp/detail/92318?lang=en}.

\bibitem{karlin1959coincidence}
Samuel Karlin and James McGregor.
\newblock Coincidence properties of birth and death processes.
\newblock {\em Pacific J. Math}, 9(4):1109--1140, 1959.

\bibitem{kent1982spectral}
John~T. Kent.
\newblock The spectral decomposition of a diffusion hitting time.
\newblock {\em The Annals of Probability}, 10(1):207--219, 1982.

\bibitem{liggett2010continuous}
Thomas~Milton Liggett.
\newblock {\em Continuous Time Markov Processes: An Introduction}.
\newblock American Mathematical Soc., 2010.

\bibitem{mandl1968analytical}
Petr Mandl.
\newblock {\em Analytical Treatment of One-dimensional Markov Processes}.
\newblock Springer, 1968.

\bibitem{rogersPitman1981markov}
L.~C.~G. Rogers and J.~W. Pitman.
\newblock Markov functions.
\newblock {\em The Annals of Probability}, 9(4):573--582, 1981.

\bibitem{swart2010intertwining}
Jan~M. Swart.
\newblock Intertwining of birth-and-death processes.
\newblock {\em Kybernetika}, 47(1):1--14, 2011.

\end{thebibliography}

\end{document}